\documentclass[12pt]{amsart}
\usepackage{
	amsmath,  amssymb,  amsthm,   amscd,
	gensymb,  graphicx, etoolbox, booktabs,
	stackrel, mathtools    
}
\usepackage[usenames,dvipsnames]{xcolor}
\definecolor{darkblue}{rgb}{0,0,0.7}
\definecolor{darkred}{rgb}{0.7,0,0}
\usepackage[colorlinks=true, linkcolor=darkred, citecolor=darkblue, urlcolor=blue, pagebackref=true, breaklinks=true]{hyperref}
\usepackage[capitalise]{cleveref}
\usepackage{enumitem}
\usepackage{placeins}
\usepackage{relsize}
\usepackage{todonotes}
\usepackage{soul}
\usepackage{tikz}
\usetikzlibrary{arrows}
\usetikzlibrary{shapes}
\usepackage{float}
\usepackage{tabularx}
\usepackage{kantlipsum}
\usepackage{array}
\usepackage[top=1.2in,bottom=1in,left=1in,right=1in]{geometry}

\newtheorem{proposition}{Proposition}[section]
\newtheorem{lemma}[proposition]{Lemma}
\newtheorem{theorem}[proposition]{Theorem}

\newtheorem{conjecture}[proposition]{Conjecture}

\tikzstyle{place}=[draw,circle,minimum size=1mm,inner sep=1pt,outer sep=-1.1pt,fill=black]

\usetikzlibrary{shapes.geometric}
\tikzstyle{places}=[draw,rectangle,minimum size=8pt,inner sep=0pt]
\tikzstyle{placesf}=[draw,rectangle,minimum size=5pt,inner sep=0pt]
\tikzstyle{placec}=[draw,circle,minimum size=8pt,inner sep=0pt]
\tikzstyle{placecf}=[draw,circle, minimum size=5pt,inner sep=0pt]

\def\K{\mathbb{K}}

\def\l{\langle}
\def\r{\rangle}

\begin{document}

\title{SIMIS and packing properties of Alexander dual of connected ideals}


 \author{Om Prakash Bhardwaj}
 \address{Chennai Mathematical Institute, Siruseri, Tamil Nadu, India}
 \email{omprakash@cmi.ac.in; opbhardwaj95@gmail.com}

 \author{Kanoy Kumar Das}
\address{Chennai Mathematical Institute, Siruseri, Tamil Nadu, India}
\email{kanoydas0296@gmail.com; kanoydas@cmi.ac.in}

 \author{Rutuja Sawant}
\address{Chennai Mathematical Institute, Siruseri, Tamil Nadu, India}
\email{rutuja@cmi.ac.in}

\keywords{Square-free monomial ideals, Symbolic powers, Packing property, Alexander dual}
\subjclass[2020]{13C05, 13F55, 05C70, 05E40}

\vspace*{-0.4cm}
\begin{abstract}
    In this article, we investigate when the ordinary and symbolic powers of the Alexander dual of connected ideals of graphs coincide, and provide a complete classification of all such graphs.  Furthermore, we prove Conforti--Cornu\`ejols conjecture for this class of~ ideals.
\end{abstract}

\maketitle

\section{Introduction}

Let $R=\K[x_1,x_2, \dots , x_n]$ be a polynomial ring over a field $\K$, and $I\subseteq R$ an ideal. The symbolic power of $I$, denoted by $I^{(s)}$, is defined by 
\[
I^{(s)} = \bigcap_{\mathfrak{p} \in \mathrm{Ass}(I)} I^s R_{\mathfrak{p}} \cap R.
\]
Symbolic powers of ideals have long been a central topic in commutative algebra and algebraic geometry. Geometrically, the $s$-th symbolic power of a prime ideal consists of all elements whose order of vanishing is at least $s$ at every closed point of the corresponding variety. Moreover, symbolic powers play an essential role in the proofs of many fundamental results in these areas.

For square-free monomial ideals, the problem of determining when the equality $I^{(s)} = I^s$ holds, admits a combinatorial interpretation in terms of the associated hypergraph. In particular, this question is closely related to a well-known conjecture of Conforti and Cornu\'{e}jols~\cite{1990Cornejols-et.al}, often referred to as the \emph{packing problem}. The original conjecture of Conforti and Cornu\'{e}jols was reformulated in the language of commutative algebra by Gitler, Valencia, and Villarreal~\cite{2005Gitler-et.al}, and can be stated as follows.

\begin{conjecture}
	For a square-free monomial ideal $I$, all the symbolic powers and ordinary powers coincide if and only if the ideal $I$ has the packing property.
\end{conjecture}

For any unexplained terminology, we refer the reader to Section~\ref{sec: preli}. One direction of the conjecture is known and is relatively easy to prove: if $I^{(s)} = I^s$ for all $s \ge 1$, then $I$ has the packing property (see \cite[Section 4]{2015Dao-et.el}). The non-trivial direction is to show that if $I$ has the packing property, then the symbolic and ordinary powers of $I$ coincide. The conjecture has been verified for several classes of ideals, including edge ideals of graphs \cite{1994Simis-et.al}, $3$-path ideals of graphs \cite{Alilooee-Banerjee}, $4$-path ideals of graphs \cite{2024Nasernejad}, complementary edge ideals \cite{2025Roy-Saha}, and matroidal ideals \cite{2025Ficarra-Moradi}. Monomial ideals for which the equality $I^{(s)} = I^s$ holds for all $s \geq 1$, are called \emph{Simis ideals}. Other than square-free monomial ideals, the Simis property has also been extensively studied for various classes of monomial ideals in the literature; see, for example, \cite{2023Banerjee-et.al, 2025Banerjee-et.al, 2024Kannoy, Gimnez-et.al2018, Grisalde.et.al2024, MandalPradhan2021, MandalPradhan2022, 2025Pinto-et.al}.

Let $G$ be a simple graph with vertex set $V(G)$ and edge set $E(G)$. The edge ideal of $G$, denoted by $I(G)$, is the square-free monomial ideal generated by the quadratic monomials corresponding to the edges of $G$. By the Stanley--Reisner correspondence, there exists a simplicial complex $\Delta(G)$ such that $I(G)=I_{\Delta(G)}$, where $I_{\Delta(G)}$ denotes the Stanley--Reisner ideal of $\Delta(G)$. The simplicial complex $\Delta(G)$ is commonly known as the independence complex of $G$. In~\cite{Paolini-Salvetti}, Paolini and Salvetti introduced the notion of higher independence complexes of a graph, which generalize the independence complex. For $t\ge 2$, the $t$-independence complex of $G$, denoted by $\mathrm{Ind}_t(G)$, consists of all subsets $A\subseteq V(G)$ such that each connected component of the induced subgraph $G[A]$ has at most $t-1$ vertices. In particular, $\mathrm{Ind}_2(G)$ coincides with the independence complex $\Delta(G)$.

The minimal monomial generators of the Stanley--Reisner ideal of $\mathrm{Ind}_t(G)$ correspond to the connected subgraphs of $G$ with $t$ vertices. Consequently, the ideal $I_t(G) \coloneqq I_{\mathrm{Ind}_t(G)}$ is referred to as the \emph{$t$-connected ideal} of $G$ \cite{2025Ananthnarayan-et.al, 2025Das-et.al}. These ideals naturally generalize edge ideals, since $I_2(G)=I(G)$, and have sparked great interest in the community, see \cite{2023Abdelmalek-et.al, 2025Deshpande-et.al, 2024Deshpande-et.al, 2022Deshpande-et.al,2021Deshpande-et.al, 2025Ghodh-et.al,2025Roy-et.al}. Moreover, for any graph $G$, the classical $3$-path ideals coincide with the $3$-connected~ ideals.

In this article, we study the Conforti--Cornu\'{e}jols conjecture for the Alexander dual of $t$-connected ideals of graphs. In \cite{Bodas-et.al}, Bodas et al., investigated how the packing property behaves under Alexander duality. They proved that if $I$ is the edge ideal of a uniform hypergraph and $I$ satisfies the packing property, then its Alexander dual $I^{\vee}$ also satisfies the packing property, while the converse does not hold in general (see~\cite[Example 5.9]{Bodas-et.al}). In \cite{Alilooee-Banerjee}, Alilooee and Banerjee characterized the packing property for $3$-path ideals and verified the Conforti--Cornu\'{e}jols conjecture in this setting. Consequently, if $I_3(G)$ satisfies the packing property, then its Alexander dual $I_3(G)^{\vee}$ also satisfies the packing property. 

Building upon these results, we give a complete answer to the Conforti--Cornu\'{e}jols conjecture for the Alexander dual of $t$-connected ideals of graphs for all $t\ge 3$. In particular, we prove the following.

\begin{theorem}[\protect{Theorem~\ref{thm: main}}]
    Let $G$ be a connected graph on $n$ vertices and let $t\geq 3$ be an integer. Let $J_t(G)$ denote the Alexander dual of the $t$-connected ideal of $G$. Then the following statements are equivalent:
    \begin{enumerate}
        \item $J_t(G)$ satisfies the packing property;
        \item $J_t(G)^{(s)}=J_t(G)^s$ for all $s\geq 1$;
        \item one of the following holds:
        \begin{enumerate}
            \item[(a)] $n=t$;
            \item[(b)] $G=P_n$ for some $n\geq t$;
            \item[(c)] $G=C_n$ and one of the following holds: \begin{itemize}
            \item[(i)] $t=3 \text{ and } n\in \{3,6,9\}$;
            \item[(ii)] $t=4 \text{ and } n\in \{4,8\}$;
            \item[(iii)] $t\geq 5 \text{ and } n=t$.
        \end{itemize}
        \end{enumerate}        
    \end{enumerate}
\end{theorem}

This paper is organized as follows. In Section~\ref{sec: preli}, we recall the necessary definitions and background results. The proof of the main theorem is given in Section~\ref{sec: main}. As an application, we show that this characterization of Simis ideals yields a solution to a linear programming duality problem.

\section{Preliminaries}\label{sec: preli}
In this section, we recall the necessary definitions and results, and fix the notations used throughout the paper. Let $\mathbb{K}$ be a field and $R = \mathbb{K}[x_1,\ldots,x_n]$ be the polynomial ring with coefficients in $\mathbb{K}$. For a monomial $m=x_1^{a_1}x_2^{a_2}\cdots x_n^{a_n}, a_i\geq 0$, the set $\{x_i \mid a_i \neq 0\}$ is called the \emph{support} of $m$, denoted by $\mathrm{Supp}(m)$. For a monomial ideal $I\subseteq R$, $\mathcal{G}(I)$ will denote the set of all minimal monomial generators of the ideal $I$. Symbolic power of the monomial ideal $I$ can be expressed using the primary decomposition of the ideal (see \cite[Theorem 3.7]{CooperEtAl2017}). In the case of square-free monomial ideals, the expression for the symbolic powers can be given as follows. Let $I\subseteq R$ be a square-free monomial ideal, and assume that $I=P_1\cap P_2\cap \cdots \cap P_k$ is the primary decomposition. Then the $s$-th symbolic power of $I$ is given by
\[
I^{(s)} = P_1^s \cap \ldots \cap P_k^s.
\]

An ideal $I$ is called \emph{Simis} if $I^{(s)} = I^s$ for all $s \geq 1$. A square-free monomial ideal $I\subseteq R$ of height $c$ is called \emph{K\"onig} if there exists a regular sequence of monomials in $I$ of length $c$. The ideal $I$ is said to have the \emph{packing property} if every ideal obtained from $I$ by setting any number of variables equal to $0$ or $1$ is K\"onig. Setting some variables equal to $0$ in an ideal is a well-known technique in commutative algebra, often referred to as \emph{restriction}. In general, the restriction operation helps identify minimal forbidden structures with respect to properties that are compatible with restriction. For square-free monomial ideals, the equality of ordinary and symbolic powers is one such property. The classical characterization of the equality of ordinary and symbolic powers of edge ideals of simple graphs uses this idea effectively: if the graph contains an odd cycle, then some symbolic power differs from the corresponding ordinary power. On the other hand, setting variables equal to $1$ can be viewed as localizing at those variables.

Let $G$ be a simple graph with vertex set $V(G)=\{x_1,\dots , x_n\}$ and edge set $E(G)$ consist of $2$-element subsets of $V(G)$. Given a subset $S\subseteq V(G)$, $G[S]$ will be the induced subgraph of $G$ on the vertex set $S$. For $t\geq 2$, the \emph{$t$-connected ideal} of $G$, denoted by $I_t(G)$, is a square-free monomial ideal defined by 
\[
I_t(G)\coloneq \l x_{i_1}\cdots x_{i_t}\mid G[\{x_{i_1},\dots , x_{i_t}\}] \text{ is connected}\r \subseteq R.
\]
The $t$-connected ideals can be regarded as the higher degree generalizations of classical edge ideals. The simplicial complex associated to $I_t(G)$ is called the \emph{$t$-independence complex.}

Our main object of study is the Alexander dual of $t$-connected ideals. Let $I$ be a square-free monomial ideal minimally generated by $\{ \mathbf{x}^{\mathbf{a}_1}, \dots , \mathbf{x}^{\mathbf{a}_r}\}$. The \emph{Alexander dual} of $I$, denoted by $I^{\vee}$, is defined as 
\[
I^{\vee}\coloneq \mathbf{m}_{\mathbf{a}_1}\cap \cdots \cap \mathbf{m}_{\mathbf{a}_r},
\]
where $\mathbf{m}_{\mathbf{a}_i}=\l x_j: x_j\mid \mathbf{x}^{\mathbf{a}_i}\r, 1\leq i\leq r$. Note that $I^{\vee}$ is again a square-free monomial ideal, whose minimal monomial generators correspond to the minimal vertex covers of the hypergraph associated with $I$. Hence, the Alexander dual of $I$ is often called the \emph{cover ideal} of $I$. Throughout this article, we write $J_t(G)\coloneqq I_t(G)^{\vee}$ for the cover ideal of the $t$-connected ideal of a graph $G$. If $x_{i_1}\cdots x_{i_r}\in \mathcal{G}(J_t(G))$, then the set $C=\{x_{i_1},\dots,x_{i_r}\}$ is a minimal vertex cover of the hypergraph associated with $I_t(G)$. We call $C$ a \emph{$t$-cover} of $G$. In the following, we recall the Simis property of the Alexander dual of edge ideals of graphs.

\begin{theorem}[{\cite[see Corollary 3.17, Theorem 4.6, Proposition 4.27]{2009Gitler-et.al}}]\label{thm:simis-coverideals-of-edgeideals}
    Let $G$ be a simple graph. Then, $J_2(G)$ is Simis if and only if $G$ is bipartite.
\end{theorem}

The following fact is well-known to experts (see, for instance, \cite{2015Dao-et.el}). After proving the equality of ordinary and symbolic powers, the following lemma is sufficient to conclude one direction of Conforti--Cornu\'{e}jols conjecture.

\begin{lemma}\label{lem: Simis implies packed}
    Let $I\subseteq R$ be a square-free monomial ideal. If $I$ is Simis then $I$ satisfies the packing property. 
\end{lemma}

\section{Main Results}\label{sec: main}

In this section, we prove the main result of this article. Let $G$ be a connected simple graph. A vertex $v\in V(G)$ is called a \emph{cut vertex} if $G\setminus v$ is disconnected; otherwise, $v$ is called a \emph{non-cut vertex}. The following lemma is one of the crucial steps in the proof of the main theorem. 
\begin{lemma}\label{lem: main}
    Let $G$ be a graph, and $H\subseteq G$ be a connected induced subgraph of $G$ such that $|V(H)|=t+1$, and $H$ has exactly $r\geq 3$ non-cut vertices. Then 
    \begin{enumerate}
        \item $J_t(G)^{(r-1)}\neq J_t(G)^{r-1}$.
        \item $J_t(G)$ is not packed.
    \end{enumerate}
\end{lemma}
\begin{proof}
    (1) Without any loss of generality, assume that $V(H)=\{x_1,\dots , x_{t+1}\}$, and among them, the non-cut vertices are $\{x_1,\dots , x_r\}$. Consider the prime ideal $P=\l x_1,x_2,\dots , x_{t+1}\r$, and take $J_t(G)R_P$. We note that 
    \[
    J_t(G)_P\coloneq J_t(G)R_P=\bigcap_{i=1}^{r} \l x_1,\dots ,x_{i-1}, \widehat{x_i}, x_{i+1},\dots x_{t+1} \r R_p.
    \]
    If we set $f=\prod_{i=1}^{r}x_i$, then $f\in \l x_1,\dots ,x_{i-1}, \hat{x_i}, x_{i+1},\dots x_{t+1} \r^{r-1}$ for all $1\leq i\leq r$, and hence, $f\in J_t(G)_P^{(r-1)}$. On the other hand, it is straightforward to see that 
    \[
    J_t(G)_P=\l x_ix_j\mid 1\leq i<j\leq r\r R_P +\l x_{r+1},\dots , x_{t+1}\r R_P.
    \]
    Since $r\geq 3$, by comparing degrees of $f$ and minimal generating set of $J_t(G)_P$ we have $f\notin J_t(G)_P^{r-1}$. Therefore, $J_t(G)_P^{(r-1)}\neq J_t(G)_P^{r-1}$, and hence, $J_t(G)^{(r-1)}\neq J_t(G)^{r-1}$ as desired.

    \medskip
    \noindent
    (2) To see that $J_t(G)$ is not packed, consider the ideal $I$ obtained by putting $x_i=1$ for all $i \notin V(H)$. Then we have, 
    \begin{align*}
        I&=\bigcap_{i=1}^{r} \l x_1,\dots ,x_{i-1}, \widehat{x_i}, x_{i+1},\dots x_{t+1} \r\\
        &= \l x_ix_j\mid 1\leq i<j\leq r\r +\l x_{r+1},\dots , x_{t+1}\r.
    \end{align*}
    Note that $\mathrm{height}(I)=t$, and $I$ has monomial regular sequence of length at most $t+1-r+\lfloor \frac{r}{2} \rfloor=t+1-\lceil \frac{r}{2} \rceil <t$, as $r\geq 3$.
\end{proof}

It is a well-known fact that any connected graph with $n\geq 2$ vertices has at least two non-cut vertices. Moreover, the only graphs with exactly two non-cut vertices are paths. Therefore, in view of the above Lemma, if $J_t(G)$ is Simis, or if $J_t(G)$ is packed, then any connected induced subgraph of $G$ on $t+1$ vertices must be a path. This suggests that the possible graphs $G$ for which $J_t(G)$ is Simis (or packed) are paths and cycles.

We begin by considering the case when $G$ is a path. Let $P_n$ denote the path graph on $n$ vertices, where 
\begin{align*}
    V(P_n)&=\{x_1,\dots , x_n\},\\
    E(P_n)&=\{x_ix_{i+1}\mid 1\leq i\leq n-1\}.
\end{align*}
We show that $J_t(G)$ is both Simis and packed for all $t\geq 1$.
The following proposition gives a complete description of the minimal generating set of the cover ideal of the $t$-connected ideal of $P_n$.

\begin{proposition}\label{prop: mingens of J_t(P_n)}
    For $1\leq i_1<i_2<\cdots < i_r\leq n$, $\mathbf{x}^{\alpha}=x_{i_1}\cdots x_{i_r}\in \mathcal{G}(J_t(P_n))$ if and only if the following holds:
    \begin{enumerate}
        \item $r\geq \lfloor \frac{n}{t} \rfloor$;
        \item $i_1\leq t, i_2\geq t+1$;
        \item $i_r\geq n-t+1, i_{r-1}\leq n-t$;
        \item $i_{j+1}-i_{j}\leq t$ for all $1\leq j \leq r-1$;
        \item $i_{j+2}-i_{j}\geq t+1$ for all $1\leq j\leq r-2$.
    \end{enumerate}
\end{proposition}
\begin{proof}
    Assume that $\mathbf{x}^{\alpha}=x_{i_1}\cdots x_{i_r}$ satisfies the given conditions. Observe that every connected induced subgraph of $P_n$ is itself a path. The conditions $(2)-(4)$ imply that after removing the vertices $x_{i_1}, \dots , x_{i_r}$ from the path $P_n$, the resulting graph does not contain any path of length $t$. In other words, the set $C=\{x_{i_1}, \dots , x_{i_r}\}$ is a $t$-cover of $P_n$. Now, to see the minimality of this $t$-cover, we will show that $C'=C\setminus x_{i_k}$ cannot form a $t$-cover of $P_n$ for any $1\leq k\leq r$. If $k=1$, or if $k=r$, then it follows from the conditions $(2)$ and $(3)$ respectively that $C'$ is not a $t$-cover of $P_n$. On the other hand, if $2\leq k\leq r-1$, then the condition $(5)$ implies that $C'$ is again not a $t$-cover of $P_n$.

    For the converse part, let us assume that $\mathbf{x}^{\alpha}=x_{i_1}\cdots x_{i_r}\in \mathcal{G}(J_t(P_n))$. Then the set of vertices $\{x_{i_1}, \dots , x_{i_r}\}$ corresponds to a minimal $t$-cover of $P_n$. It is now easy to verify that any minimal $t$-cover satisfies the conditions given in $(1)-(5)$.
\end{proof}

\begin{proposition}\label{prop: Symbolic Prop for P_n}
    For $n\geq t\geq 2$, we have $J_t(P_n)^{(s)}= J_t(P_n)^{s}$ for all $s\geq 1$. In particular, $J_t(P_n)$ has the packing property. 
\end{proposition}
\begin{proof} 
We only need to show that $J_t(P_n)^{(s)}\subseteq J_t(P_n)^{s}$ for all $s\geq 1$. We proceed by induction on $n+s$. Note that the assertion is always true if $s=1$, or if $n=t$. For the later case, $J_t(P_n)$ is the prime ideal generated by the variables, and hence a complete intersection ideal. We now assume that $n+s>t+1$. Let $f\in J_t(P_n)^{(s)}$. We can write $J_t(P_n)^{(s)}= \l x_1,\ldots ,x_t\r^s\cap J_t(P_n\setminus x_1)^{(s)}$. Then we have $f\in J_t(P_n\setminus x_1)^{(s)}=J_t(P_n\setminus x_1)^{s}$, where the last equality follows from the induction hypothesis. Consequently, there is some $g\in \mathcal{G}(J_t(P_n\setminus x_1))$ such that $g\mid f$ and $\frac{f}{g}\in J_t(P_n\setminus x_1)^{s-1}$. We divide the proof into two cases:
	
	\noindent
	\textsc{Case 1}: Assume that $x_1\mid f$.  We further consider the following subcases:
	
	\noindent
	\textsc{Subcase 1(A)}: Assume that $\mathrm{Supp}(g)\cap \{x_2, \dots , x_t\}\neq \emptyset$. Since $g\in \mathcal{G}(J_t(P_n\setminus x_1))$, it follows from Proposition~\ref{prop: mingens of J_t(P_n)}(2) that $|\mathrm{Supp}(g)\cap \{x_2, \dots , x_t\}|=1$, and hence, $g\in \mathcal{G}(J_t(P_n))$. Also, since $x_1\nmid g$, we have $\frac{f}{g}\in \l x_1,\dots ,x_t\r^{s-1}$. Moreover, $\frac{f}{g}\in J_t(P_n\setminus x_1)^{s-1}=J_t(P_n\setminus x_1)^{(s-1)}$ by the induction hypothesis. This implies that $\frac{f}{g}\in  \l x_1, \dots ,x_t\r^{s-1}\cap J_t(P_n\setminus x_1)^{(s-1)}= J_t(P_n)^{(s-1)}=J_t(P_n)^{s-1}$, where the last equality again follows from the induction hypothesis. Finally, since $g\in \mathcal{G}(J_t(P_n))$ we have $f\in J_t(P_n)^{s}$.
	
	\noindent
	\textsc{Subcase 1(B)}: Assume that $\mathrm{Supp}(g)\cap \{x_2, \dots, x_t\}= \emptyset$. Then by Proposition~\ref{prop: mingens of J_t(P_n)}, $x_{t+1}\mid g$. In this case, we take $g'=x_1g$. Since $x_1\mid f$, $g'\mid f$, and moreover, $\frac{f}{g'}\in \l x_1,\dots , x_t\r^{s-1}$. On the other hand, since $\frac{f}{g}\in  J_t(P_n\setminus x_1)^{s-1}$, we have  $\frac{f}{g'}\in  J_t(P_n\setminus x_1)^{s-1}= J_t(P_n\setminus x_1)^{(s-1)}$, where the last equality uses the induction hypothesis. Hence, we get $\frac{f}{g'}\in \l x_1,\dots , x_t\r^{s-1}\cap J_t(P_n\setminus x_1)^{(s-1)}= J_t(P_n)^{(s-1)}= J_t(P_n)^{s-1}$, where the last equality again follows from the induction hypothesis. By Proposition~\ref{prop: mingens of J_t(P_n)}, $g'\in J_t(P_n)$, and therefore, $f\in J_t(P_n)^{s}$ as desired.
	
	\medskip
	\noindent
	\textsc{Case 2}: Assume that $x_1\nmid f$. If $\mathrm{Supp}(g)\cap \{x_2,\ldots ,x_t\}\neq \emptyset$, then by similar arguments as in Subcase 1(A), we can conclude that $g\in J_t(P_n)$, and $\frac{f}{g}\in J_t(P_n)^{s-1}$. Thus, we have $f\in J_t(P_n)^{s}$. Finally, assume that $\mathrm{Supp}(g)\cap \{x_2,\dots ,x_t\}= \emptyset$. Since $x_1\nmid f$ and $f\in \l x_1,\dots ,x_t\r^s$, there are integers $a_i\geq 0, 2\leq i\leq t$ such that $x_2^{a_2}x_3^{a_3}\cdots x_{t}^{a_t}\mid f$ and $\sum_{i=2}^ta_i=s$. Now, consider the smallest $\ell, 2\leq \ell \leq t$ such that $a_{\ell}\neq 0$, and take $g'=x_{\ell}g$. Note that, by Proposition~\ref{prop: mingens of J_t(P_n)} we have $g'\in J_t(P_n)$. Further, since $\mathrm{Supp}(g)\cap \{x_1,\dots ,x_t\}= \emptyset$, and $f\in J_t(P_n)^{(s)}$, we have $\frac{f}{g'}\in \l x_i, x_{i+1},\dots , x_{i+t-1}\r^{s-1}$ for all $1\leq i\leq \ell$. Moreover, since $\frac{f}{g}\in J_t(P_n\setminus x_1)^{(s-1)}\subseteq J_t(P_n\setminus \{x_1,\dots , x_{\ell}\})^{(s-1)}$, we have $\frac{f}{g'}\in J_t(P_n\setminus \{x_1,\dots , x_{\ell}\})^{(s-1)}$. Therefore, $\frac{f}{g'}\in \bigcap_{i=1}^{\ell} \l x_i, x_{i+1},\dots , x_{i+t-1}\r^{s-1} \cap J_t(P_n\setminus \{x_1,\dots , x_{\ell}\})^{(s-1)}= J_t(P_n)^{(s-1)}= J_t(P_n)^{s-1}$, where the last equality follows from the induction hypothesis. Since $g'\in J_t(P_n)$, it follows that $f\in J_t(P_n)^s$, and this completes the proof.
\end{proof}

Next, we consider the $t$-connected ideals of cycles. Let $C_n$ denote the cycle on $n$ vertices,~ where 
\begin{align*}
    V(C_n)&=\{x_1,\dots , x_n\},\\
    E(C_n)&=\{x_ix_{i+1}\mid 1\leq i\leq n-1\}\cup \{x_1x_n\}.
\end{align*}
The following proposition provides a complete description of minimal generating set of~ $J_t(C_n)$.

\begin{proposition}\label{prop: mingens of J_t(C_n)}
	For $1\leq i_1<i_2<\cdots < i_r\leq n$, $\mathbf{x}^{\alpha}=x_{i_1}\cdots x_{i_r}\in \mathcal{G}(J_t(C_n))$ if and only if the following holds:
	\begin{enumerate}
		\item $r\geq \lceil \frac{n}{t} \rceil$;
		\item $i_{j+1}-i_{j}\leq t$ for all $1\leq j \leq r-1$, $n+i_{1}-i_{r}\leq t$;
		\item $i_{j+2}-i_{j}\geq t+1$ for all $1\leq j\leq r-2$, $n+i_{1}-i_{r-1}\geq t+1$, $n+i_{2}-i_{r}\geq t+1$.
	\end{enumerate}
\end{proposition}
\begin{proof}
Assume that $\mathbf{x^{\alpha}} = x_{i_1}\ldots x_{i_r}$ satisfies the given conditions. Condition $(2)$ implies that after removing the vertices $x_{i_1},\ldots,x_{i_r}$ from $C_n$, the resulting graph does not contain any path of length $t$. In other words, the set $C = \{x_{i_1},\ldots,x_{i_r}\}$ is a $t$-cover of $C_n$. To see the minimality of the $t$-cover, we take $C' = C\setminus \{x_{i_k}\}$, where $1\leq k\leq r$. Indeed, it follows from condition (3) that $C'$ is not a $t$-cover of $C_n$. Hence, $C$ is a minimal $t$-cover of $C_n$, and therefore $\mathbf{x}^{\alpha}\in \mathcal{G}(J_t(G))$.

Conversely, suppose that $\mathbf{x}^{\alpha}=x_{i_1}\cdots x_{i_r}\in \mathcal{G}(J_t(C_n))$. Then the set $\{x_{i_1},\dots,x_{i_r}\}$ corresponds to a minimal $t$-cover of the cycle $C_n$. Since any gap of more than $t$ vertices between consecutive chosen vertices would contain a $t$-path, the distances between consecutive indices must satisfy the inequalities in $(2)$. Moreover, minimality of the cover implies that removing any vertex must create a segment of length at least $t$, which yields the inequalities in $(3)$. Finally, covering a cycle of length $n$ by gaps of size at most $t$ requires at least $\lceil n/t\rceil$ vertices, giving condition $(1)$. TTherefore, conditions $(1)-(3)$ hold.
\end{proof}

\begin{proposition}\label{prop:Konig}
    $J_t(C_n)$ is K\"onig if and only if $n=t\ell $ for some $\ell\geq 1$.
\end{proposition}
\begin{proof}
  Assume that $n=t\ell$ for some $\ell\geq 1$. By Proposition~\ref{prop: mingens of J_t(C_n)}, $f_i\coloneqq x_ix_{i+t}\cdots x_{i+(\ell-1)t}\in J_t(C_n)$ for all $1\leq i\leq t$. Note that $\mathrm{height}(J_t(C_n))=t$ and $f_1,\dots , f_t$ forms a monomial regular sequence. Therefore, $J_t(C_n)$ is K\"onig. Conversely, assume that $n=t\ell + m$ for some $1\leq m\leq t-1$. If possible, let us assume that there is a monomial regular sequence $g_1, \dots , g_t\in \mathcal{G}(J_t(C_n))$. By Proposition~\ref{prop: mingens of J_t(C_n)}, $\mathrm{deg}(g_i)\geq \lceil\frac{n}{t}\rceil=\ell +1$ for all $1\leq i\leq t$. Since $\mathrm{Supp}(g_i)\cap \mathrm{Supp}(g_j)=\emptyset$ for all $i\neq j\in [t]$, we have $\sum_{i=1}^{t}|\mathrm{Supp}(g_i)|\geq t(\ell +1)=t\ell +t>n$, a contradiction. Therefore, $J_t(C_n)$ is not K\"onig.  
\end{proof}

\begin{lemma}\label{lem:cycle-not-packed}
    Let $n>t$ be positive integers such that $(n,t) \notin \{(6,3),(9,3), (8,4)\}$. Then $J_t(C_n)$ is not packed.    
\end{lemma}
\begin{proof}
     Note that if $n \neq \ell t$, then by Proposition~\ref{prop:Konig} it follows that $J_t(C_n)$ is not K\"onig, and therefore $J_t(C_n)$ is not packed. Hence, for the remainder of the proof we may assume that $n=\ell t$ for some integer $\ell \geq 2$. We first consider the case $t=3$. Suppose that $n=3\ell$ with $\ell \geq 4$. In this case the ideal $J_3(C_n)$ has the form
\begin{equation}\label{eq:coverideal3path}
J_3(C_n)
= \langle x_{n-1},x_n,x_1 \rangle \cap \langle x_n,x_1,x_2 \rangle
\cap \left( \bigcap_{i=1}^{n-2} \langle x_i,x_{i+1},x_{i+2} \rangle \right).
\end{equation}

Our strategy is to set suitable variables equal to zero so that the resulting ideal becomes the cover ideal of an odd cycle. Theorem~\ref{thm:simis-coverideals-of-edgeideals} then implies that $J_3(C_n)$ is not packed. We proceed by considering the following cases.

\noindent
\textsc{Case 1}: $n=12k$ with $k\ge1$.

\noindent
\textsc{Subcase 1(A):} Suppose $k$ is even.
Set $x_1=x_4=x_7=x_{10}=0$ and set $x_{4i+13}=0$ for $0\le i\le 3k-4$.
Substituting these values into Equation~(\ref{eq:coverideal3path}) gives
\[
\begin{aligned}
I
&= \langle x_n,x_2 \rangle \cap \langle x_2,x_3 \rangle \cap \langle x_3,x_5 \rangle \cap \langle x_5,x_6 \rangle
   \cap \langle x_6,x_8 \rangle \cap \langle x_8,x_9 \rangle \\
&\qquad \cap \langle x_9,x_{11} \rangle \cap \langle x_{11},x_{12} \rangle \cap  \langle x_{12},x_{14} \rangle \\
&\qquad \cap
\left( \bigcap_{i=0}^{3k-5}
\left(\langle x_{4i+14},x_{4i+15} \rangle \cap \langle x_{4i+15},x_{4i+16} \rangle \cap \langle x_{4i+16},x_{4i+18} \rangle \right)\right) \\
&\qquad \cap \langle x_{12k-2},x_{12k-1} \rangle \cap \langle x_{12k-1},x_{12k}\rangle,
\end{aligned}
\]
which is the cover ideal of $C_{9k-1}$. Since $k$ is even, $9k-1$ is odd.

\noindent
\medskip
\textsc{Subcase 1(B):} Suppose $k$ is odd.
Set $x_{4i+1}=0$ for $0\le i\le 3k-1$. Then
\[
I
= \langle x_n,x_2 \rangle
\cap \left(\bigcap_{i=0}^{3k-2}
\left(\langle x_{4i+2},x_{4i+3} \rangle \cap \langle x_{4i+3},x_{4i+4} \rangle \cap \langle x_{4i+4},x_{4i+6} \rangle \right)\right)
\cap \langle x_{12k-1},x_{12k} \rangle,
\]
which is the cover ideal of $C_{9k}$. Since $k$ is odd, $9k$ is odd.

\medskip
\noindent
\textsc{Case 2}: Suppose $n=12k+3$ with $k\ge1$. Set $x_1=x_5=x_9=0$ and set $x_{3i+13}=0$ for $0\le i\le 4k-4$.
Substituting into Equation~(\ref{eq:coverideal3path}) yields
\[
\begin{aligned}
I
&= \langle x_n,x_2 \rangle \cap \langle x_2,x_3 \rangle \cap \langle x_3,x_4 \rangle
   \cap \langle x_4,x_6 \rangle \cap \langle x_6,x_7 \rangle \\
&\qquad \cap \langle x_7,x_8 \rangle \cap \langle x_8,x_{10} \rangle \cap \langle x_{10},x_{11} \rangle
   \cap \langle x_{11},x_{12} \rangle \cap \langle x_{12},x_{14} \rangle \\
&\qquad \cap
\left(\bigcap_{i=0}^{4k-5}
\left( \langle x_{3i+14},x_{3i+15} \rangle \cap \langle x_{3i+15},x_{3i+17} \rangle \right)\right)
\cap \langle x_{12k+2},x_{12k+3} \rangle,
\end{aligned}
\]
which is the cover ideal of the odd cycle $C_{8k+3}$.

\medskip
\noindent
\textsc{Case 3:} Suppose $n=12k+6$ with $k\ge1$. Set $x_{3i+1}=0$ for $0\le i\le 4k-2$ and additionally set $x_{12k-1}=x_{12k+3}=0$. Then Equation~(\ref{eq:coverideal3path}) yields
\[
\begin{aligned}
I
&= \left(\bigcap_{i=0}^{4k-3}
\left(\langle x_{3i+2},x_{3i+3} \rangle \cap \langle x_{3i+3},x_{3i+5} \rangle \right)\right)
\cap \langle x_{12k-4},x_{12k-3} \rangle \cap \langle x_{12k-3},x_{12k-2}\rangle \\
&\qquad \cap \langle x_{12k-2},x_{12k} \rangle \cap \langle x_{12k},x_{12k+1}\rangle
\cap \langle x_{12k+1},x_{12k+2} \rangle \cap \langle x_{12k+2},x_{12k+4} \rangle \\
&\qquad \cap \langle x_{12k+4},x_{12k+5} \rangle \cap \langle x_{12k+5},x_{12k+6} \rangle
\cap \langle x_{12k+6},x_2 \rangle,
\end{aligned}
\]
which is the cover ideal of the odd cycle $C_{8k+5}$.

\medskip
\noindent
\textsc{Case 4:} Suppose $n=12k+9$ with $k\ge1$. Set $x_{3i+1}=0$ for $0\le i\le 4k-1$ and additionally set $x_{12k+2}=x_{12k+6}=0$. Then Equation~(\ref{eq:coverideal3path}) yields
\[
\begin{aligned}
I
&= \left(\bigcap_{i=0}^{4k-2}
\left(\langle x_{3i+2},x_{3i+3}) \cap \langle x_{3i+3},x_{3i+5}\rangle\right)\right)
\cap \langle x_{12k-1},x_{12k} \rangle \cap \langle x_{12k},x_{12k+1} \rangle \\
&\qquad \cap \langle x_{12k+1},x_{12k+3} \rangle \cap \langle x_{12k+3},x_{12k+4}\rangle
\cap \langle x_{12k+4},x_{12k+5} \rangle \cap \langle x_{12k+5},x_{12k+7}\rangle \\
&\qquad \cap \langle x_{12k+7},x_{12k+8}\rangle \cap (x_{12k+8},x_{12k+9})
\cap \langle x_{12k+9},x_2 \rangle,
\end{aligned}
\]
which is the cover ideal of a cycle of length ${8k+7}$, and this completes the proof for $t=3$.

Next, consider the case $t=4$. Let $\ell \ge 4$ and write $n=4\ell$. We have 
\[
J_4(C_{4\ell})= \langle x_{n-2},x_{n-1},x_n,x_1 \rangle \cap \langle x_{n-1}x_{n},x_1,x_2 \rangle
\cap \langle x_{n}x_{1},x_2,x_3 \rangle
\cap \left( \bigcap_{i=1}^{n-3} \langle x_i,x_{i+1},x_{i+2}, x_{i+2} \rangle \right).
\]
Let $I$ be the ideal obtained from $J_4(C_{4\ell})$ by setting
$x_{4m}=0$ for all $1\le m\le \ell$. Then
\[
\begin{aligned}
I = \langle x_1,x_2,x_3\rangle
\cap \langle x_2,x_3,x_5\rangle
& \cap \langle x_3,x_5,x_6\rangle
\cap \cdots
\cap \langle x_{4\ell-3},x_{4\ell-2},x_{4\ell-1}\rangle \\
&\qquad \cap \langle x_{4\ell-2},x_{4\ell-1},x_1\rangle
\cap \langle x_{4\ell-1},x_1,x_2\rangle .
\end{aligned}
\]
Observe that $I$ can be identified with the cover ideal of the
$3$-connected ideal of a cycle of length $3\ell$. Indeed, this cycle is obtained from $C_{4\ell}$ by removing the vertices $x_{4m}$,
$1\le m\le \ell$, and joining the two neighbors of $x_{4m}$ by an edge. Since $\ell\ge4$, it follows from the previous discussion that $I$ is not packed. Consequently, $J_4(C_{4\ell})$ is not packed for $\ell\ge4$. Now, consider $J_4(C_{12})$. Setting $x_1=x_4=x_6=x_8=x_{11}=0$ yields an ideal $I$ that can be identified with the cover ideal of the edge ideal of the cycle $C_7$. Since $J_2(C_7)$ is not packed, it follows that $J_4(C_{12})$ is also not packed.

Finally, we show that $J_t(C_{t \ell })$ is not packed for all $t\ge 5$ and $\ell\ge 2$ by induction on $t$. First consider the case $t=5$. Let $\ell\ge 3$ and obtain an ideal $I$ from $J_5(C_{5\ell})$ by setting $x_{5m}=0$ for all $1\le m\le \ell$. As in the previous arguments, the ideal $I$ can be identified with the cover ideal of the $4$-connected ideal of a cycle of length $4\ell$. Since $\ell\ge 3$, it follows from the above discussion that $I$ is not packed. Consequently, $J_5(C_{5\ell})$ is not packed for $\ell\ge 3$. For $J_5(C_{10})$, set $x_{2r}=0$ for $1\le r\le 5$. The resulting ideal $I$ can be identified with the cover ideal of the edge ideal of a $5$-cycle, which is not packed. Hence, $J_5(C_{10})$ is not packed.

Now assume $t\ge 6$ and that the assertion holds for all integers $t'<t$. Let $n=\ell t$ with $\ell\ge 2$, and obtain an ideal $I$ from $J_t(C_{t \ell})$ by setting $x_{tm}=0$ for all $1\le m\le \ell$. As before, the resulting ideal $I$ can be identified with the cover ideal of the $(t-1)$-connected ideal of a cycle of length ${(t-1)\ell}$. By the induction hypothesis, $I$ is not packed. Therefore, $J_t(C_{t\ell})$ is not packed.
\end{proof}

\begin{proposition}\label{prop:symbolicpowers-cycles}
    The following statements are true:
    \begin{enumerate}
        \item $J_t(C_n)^{(s)}= J_t(C_n)^{s}$ for all $s\geq 1$ if and only if 
\begin{enumerate}
            \item $t=3, n=3,6,9$;
            
            \item $t=4, n=4,8$;
            
            \item for $t\geq 5$, $n=t$.
        \end{enumerate}
        
        \item $J_t(C_n)^{(s)}= J_t(C_n)^{s}$ for all $s\geq 1$ if and only if $J_t(C_n)$ is packed.
    \end{enumerate}    
\end{proposition}
\begin{proof}
    \medskip \noindent
    (1) For the `if' part, observe that when $n=t$, the ideal $J_t(C_t)$ is generated by variables. Consequently, $J_t(C_t)^{(s)} = J_t(C_t)^s$ for all $s \ge 1$. Now assume that $n \ne t$. We treat the remaining cases for $n$ and $t$ separately.

    We first consider the case $n=6$ and $t=3$. We show that $J_3(C_6)^{(s)} = J_3(C_6)^s$ for all $s \ge 1$. The proof proceeds by induction on $s$. The assertion is clear for $s=1$. Assume that $J_3(C_6)^{(\ell)} = J_3(C_6)^{\ell}$ for every $\ell < s$. It therefore suffices to prove that $J_3(C_6)^{(s)} \subseteq J_3(C_6)^s.$ Let $f \in J_3(C_6)^{(s)}$. Without loss of generality, assume that $x_6 \mid f$. From the description of symbolic powers, it follows that
    \[
    J_3(C_6)^{(s)} =\l x_4,x_5,x_6 \r^s
    \cap \l x_5,x_6,x_1\r^s
    \cap \l x_6,x_1,x_2 \r^s
    \cap J_3(P_5)^{(s)}.
    \]
    In particular, $f \in J_3(P_5)^{(s)}$. Hence, by Proposition~\ref{prop: Symbolic Prop for P_n}, there exists $g \in \mathcal{G}(J_3(P_5))$ such that $g \mid f$ and $\frac{f}{g} \in J_3(P_5)^{s-1}= J_3(P_5)^{(s-1)}.$ We consider all possible choices of the generator $g$.
    \begin{enumerate}
        \item If $g = x_1x_4$ or $g = x_2x_5$, then $g \in \mathcal{G}(J_3(C_6))$ and $g \mid f$. Note that $\frac{f}{g} \in J_3(C_6)^{(s-1)} = J_3(C_6)^{s-1}$, where the last equality follows from the induction hypothesis. Hence $f \in J_3(C_6)^s$.

        \item If $g = x_3$, set $g' = x_3x_6$. Then $g' \in \mathcal{G}(J_3(C_6))$ and $g' \mid f$. Again, observe that $\frac{f}{g'} \in J_3(C_6)^{(s-1)} = J_3(C_6)^{s-1}$, where the last equality comes from the induction hypothesis. Thus $f \in J_3(C_6)^s$.

        \item If $g = x_2x_4$, we may assume that $x_1 \nmid f$, $x_3 \nmid f$, and $x_5 \nmid f$; otherwise one of the previous cases would apply. Hence $x_6^s \mid f$. If we set $g' = x_2x_4x_6$, then observe that $\frac{f}{g'} \in J_3(C_6)^{(s-1)} = J_3(C_6)^{s-1}$, and therefore $f \in J_3(C_6)^s$.
    \end{enumerate}
    Thus $J_3(C_6)^{(s)} \subseteq J_3(C_6)^s$, and hence, $J_3(C_6)^{(s)} = J_3(C_6)^s$ for all $s \ge 1$.

Next, we consider the case $n=9$ and $t=3$. To show $J_3(C_9)^{(s)} = J_3(C_9)^s$ for all $s \ge 1$, we again proceed by induction on $s$. The statement is true for $s=1$. Assume that $J_3(C_9)^{(\ell)} = J_3(C_9)^{\ell}$ for every $\ell<s$. It now suffices to prove that $J_3(C_9)^{(s)} \subseteq J_3(C_9)^s.$ Let $f\in J_3(C_9)^{(s)}$ and assume $x_9\mid f$. From the expression of symbolic powers, we get
\[
J_3(C_9)^{(s)} = \l x_7,x_8,x_9 \r ^s
\cap \l x_8,x_9,x_1 \r ^s
\cap \l x_9,x_1,x_2\r ^s
\cap J_3(P_8)^{(s)}.
\]
Hence, by Proposition~\ref{prop: Symbolic Prop for P_n}, there exists
$g\in \mathcal{G}(J_3(P_8))$ such that
$g\mid f$ and $\frac{f}{g}\in J_3(P_8)^{s-1}
=J_3(P_8)^{(s-1)}.$ By Proposition~\ref{prop: mingens of J_t(P_n)}, we have
\begin{align*}   
\mathcal{G}(J_3(P_8)) = \left\lbrace
\begin{array}{c}
    x_1x_4x_5x_8, x_1x_4x_6, x_1x_4x_7, x_2x_4x_6,x_2x_4x_7,x_2x_5x_6, \\ [1mm]
    x_2x_5x_7,x_2x_5x_8,x_3x_4x_7,x_3x_5x_7,x_3x_5x_8,x_3x_6
\end{array}
\right\rbrace .
\end{align*}
We now consider the possible choices for the generator $g$ case by case.

\begin{enumerate}

\item If $g=x_1x_4x_7$ or $g=x_2x_5x_8$, then
$g\in \mathcal{G}(J_3(C_9))$.
Thus $g\mid f$ and moreover,
$\frac{f}{g}\in J_3(C_9)^{(s-1)}=J_3(C_9)^{s-1}$, where the last equality is by induction. Therefore, $f\in J_3(C_9)^s$.

\item If $g=x_3x_6$, set $g'=x_3x_6x_9$.
Then $g'\in \mathcal{G}(J_3(C_9))$ and $g'\mid f$.
Note that $\frac{f}{g'}\in J_3(C_9)^{(s-1)}$, and hence by the induction hypothesis, $\frac{f}{g'}\in J_3(C_9)^{s-1}$. Therefore, $f\in J_3(C_9)^s$.

\item If $g=x_1x_4x_6$, we may assume
$x_3\nmid f$ and $x_7\nmid f$; otherwise one of the previous cases applies. It then follows that
$x_1^{\alpha_1}x_2^{\alpha_2}\mid f$ with $\alpha_1+\alpha_2\ge s$ and $x_8^{\beta_1}x_2^{\beta_2}\mid f$ with $\beta_1+\beta_2\ge s$. Let $g'=x_1x_4x_6x_9\in J_3(C_9)$. Then observe that $\frac{f}{g'}\in J_3(C_9)^{(s-1)}=\frac{f}{g'}\in J_3(C_9)^{s-1}$, where the equality follows by induction. Hence, $f\in J_3(C_9)^s$. A similar argument applies to the cases when  $g\in \{ x_3x_5x_8, x_2x_4x_6, x_3x_5x_7, x_2x_5x_6, x_3x_4x_7, x_2x_4x_7, x_2x_5x_7\}$.

\item If $g=x_1x_4x_5x_8$, we may assume that $x_2,x_3, x_6,x_7\nmid f$; otherwise, one of the previous cases applies. It then follows that $x_1^s\mid f$ and $x_8^s\mid f$. Hence, $\frac{f}{g}\in J_3(C_9)^{(s-1)}=J_3(C_9)^{s-1}$, where the equality follows from the induction hypothesis.
Since $g\in \mathcal{G}(J_3(C_9))$, we conclude that
$f\in J_3(C_9)^s$.

\end{enumerate}
Thus, in any case $f\in J_3(C_9)^s$, and therefore $J_3(C_9)^{(s)} = J_3(C_9)^s$ for all $s\ge1$.

Finally, consider the case $n=8$ and $t=4$. We prove that
$J_4(C_8)^{(s)} = J_4(C_8)^s$ for all $s \ge 1$ by induction on $s$. Assume that $J_4(C_8)^{(\ell)} = J_4(C_8)^{\ell}$ for every $1\le \ell < s$. Without loss of generality, assume that $f \in J_4(C_8)^{(s)}$ with $x_8 \mid f$. From the expressions of symbolic powers, we have
\[
J_4(C_8)^{(s)}= \l x_5,x_6,x_7,x_8\r ^s
\cap \l x_6,x_7,x_8,x_1\r ^s
\cap \l x_7,x_8,x_1,x_2\r ^s
\cap \l x_8,x_1,x_2,x_3\r ^s
\cap J_4(P_7)^{(s)}.
\]
Since $f \in J_4(P_7)^{(s)}$, by Proposition~\ref{prop: Symbolic Prop for P_n}, $f \in J_4(P_7)^{s}$, and hence, there exists $g \in \mathcal{G}(J_4(P_7))$ such that $g \mid f$ and $\frac{f}{g} \in J_4(P_7)^{s-1}$. By Proposition~\ref{prop: mingens of J_t(P_n)}, we have 
\[
\mathcal{G}(J_4(P_7))=\{x_1x_5,\; x_2x_5,\; x_2x_6,\; x_3x_5,\; x_3x_6,\; x_3x_7,\; x_4\}.
\]
We now consider the possible choices for the generator $g$ case by case.
\begin{enumerate}

\item If $g \in \{x_1x_5, x_2x_6, x_3x_7\}$, then
$g \in \mathcal{G}(J_4(C_8))$. Observe that $\frac{f}{g} \in J_4(C_8)^{(s-1)} = J_4(C_8)^{s-1}$, where the equality follows from the induction hypothesis. Hence $f \in J_4(C_8)^s$.

\item If $g = x_4$, take $g' = x_4x_8$.
Then $g' \in \mathcal{G}(J_4(C_8))$ and $g' \mid f$.
Moreover, $\frac{f}{g'} \in J_4(C_8)^{(s-1)} = J_4(C_8)^{s-1}$,
so $f \in J_4(C_8)^s$.

\item If $g = x_2x_5$, we may assume that
$x_1 \nmid f$, $x_4 \nmid f$, and $x_6 \nmid f$;
otherwise one of the previous cases applies.
Then $x_5^{\alpha_1}x_7^{\alpha_2} \mid f$ with $\alpha_1+\alpha_2 \ge s$,
$x_7^{\beta_1}x_8^{\beta_2} \mid f$ with $\beta_1+\beta_2 \ge s$,
and $x_2^{\gamma_1}x_3^{\gamma_2} \mid f$ with $\gamma_1+\gamma_2 \ge s$.
Let $g' = x_2x_5x_8$. Then
$\frac{f}{g'} \in J_4(C_8)^{(s-1)} = J_4(C_8)^{s-1}$,
and hence $f \in J_4(C_8)^s$.

\item If $g = x_3x_5$, we may assume that
$x_1,x_2,x_4,x_7 \nmid f$; otherwise a previous case applies.
Then $x_8^s \mid f$.
Let $g' = x_3x_5x_8$. Hence
$\frac{f}{g'} \in J_4(C_8)^{(s-1)} = J_4(C_8)^{s-1}$,
and therefore $f \in J_4(C_8)^s$.

\item If $g = x_3x_6$, we may assume that
$x_2,x_4,x_5,x_7 \nmid f$; otherwise one of the earlier cases applies.
Thus $x_3^s \mid f$ and $x_6^s \mid f$.
Let $g' = x_3x_5x_8$. Then
$\frac{f}{g'} \in J_4(C_8)^{(s-1)} = J_4(C_8)^{s-1}$,
so $f \in J_4(C_8)^s$.

\end{enumerate}
Therefore, in any case $f\in J_4(C_8)^s$. Hence, we can conclude that $J_4(C_8)^{(s)} = J_4(C_8)^s$ for all $s \ge 1$. This completes the proof for the `if' part. The `only if' part follows from Lemma~\ref{lem:cycle-not-packed}
and Lemma~\ref{lem: Simis implies packed}.

\medskip 
\noindent
(2) Follows from (1) and Lemma~\ref{lem:cycle-not-packed}.
\end{proof}

Let $G$ be a connected graph on $n$ vertices, and assume that $t = n$. Then $I_t(G) \subseteq R$ is a principal ideal generated by the product of all the variables. Consequently, $J_t(G)$ coincides with the maximal ideal $\l x_1,\dots , x_n\r$. In particular, $J_t(G)$ is Simis and satisfies the packing property.  The main result now follows by combining Lemma~\ref{lem: main}, Proposition~\ref{prop: Symbolic Prop for P_n}, and Proposition~\ref{prop:symbolicpowers-cycles}.

\begin{theorem}\label{thm: main}
    Let $G$ be a connected graph on $n$ vertices and let $t\geq 3$ be an integer. Let $J_t(G)$ denote the Alexander dual of the $t$-connected ideal of $G$. Then the following statements are equivalent:
    \begin{enumerate}
        \item $J_t(G)$ satisfies the packing property;
        \item $J_t(G)^{(s)}=J_t(G)^s$ for all $s\geq 1$;
        \item one of the following holds:
        \begin{enumerate}
            \item[(a)] $n=t$;
            \item[(b)] $G=P_n$ for some $n\geq t$;
            \item[(c)] $G=C_n$ and one of the following holds: \begin{itemize}
            \item[(i)] $t=3 \text{ and } n\in \{3,6,9\}$;
            \item[(ii)] $t=4 \text{ and } n\in \{4,8\}$;
            \item[(iii)] $t\geq 5 \text{ and } n=t$.
        \end{itemize}
        \end{enumerate}        
    \end{enumerate}
\end{theorem}

\medskip

\subsection*{Applications to Linear Optimization}

Let $I=\l \mathbf{x}^{\mathbf{a}_1}, \dots , \mathbf{x}^{\mathbf{a}_m}\r\subseteq \K[x_1,\dots , x_n]$ be a square-free monomial ideal, where $\mathbf{a}_i\in \{0,1\}^n$. Let $A$ be the $n\times m$ matrix with entries equal to $0$ or $1$, such that its columns are the vectors $\mathbf{a}_1, \mathbf{a}_2, \dots ,\mathbf{a}_m$. Consider the integer linear program given by
\begin{align}\label{eqn: finding dual}
    A^T\mathbf{x}\geq 1, \mathbf{x}\in \{0,1\}^n.
\end{align}
Note that the above ILP always has a solution. We are interested in the minimal solutions of the above system. Note that a minimal feasible solution $\mathbf{x}$ (i.e., for any $\mathbf{w} \in \{0,1\}^n$, if $\mathbf{x} - \mathbf{w} \in \{0,1\}^n$ then $\mathbf{x} - \mathbf{w}$ is not a feasible solution) of the above system corresponds to a minimal monomial generator of the Alexander dual of the ideal $I$.  Let $\{\mathbf{b}_1, \mathbf{b}_2, \dots ,\mathbf{b}_r\}$ be the set of all minimal feasible solutions of the system given in Equation~(\ref{eqn: finding dual}). Consider the $n\times r$ matrix $B$ such that its column vectors are the vectors $\mathbf{b}_1,\mathbf{b}_2,\dots ,\mathbf{b}_r$. Consider the following linear programming problems:
\begin{equation}
    \tag{3}
    \begin{aligned}
        \text{minimize} \quad & \alpha \cdot \mathbf{y} \\
        \text{subject to} \quad & B^T\mathbf{y} \ge \mathbf{1}, \; \mathbf{y} \in \mathbb{N}^n .
    \end{aligned}
\end{equation}
    \text{and}
\begin{equation}
    \tag{4}
    \begin{aligned}
        \text{maximize} \quad & \mathbf{z} \cdot \mathbf{1} \\
        \text{subject to} \quad & B\mathbf{z} \le \alpha, \; \mathbf{z} \in \mathbb{N}^r.
    \end{aligned}
\end{equation}
where $\alpha\in \mathbb{N}^n$, and $\mathbf{1}$ is the vector in $\mathbb{N}^r$ with each entry equal to $1$. Let $\tau_{\alpha}(B)$ and $\nu_{\alpha}(B)$ be the optimal values for the above linear programming problems, respectively. With the notations and terminologies given above, and as a consequence of our main theorem, we have the following.

\begin{theorem}
    Let $A=M_t(G)$, the incidence matrix corresponding to a $t$-connected ideal of a connected graph $G$, and let $B$ be the matrix corresponding to the minimal feasible solution of the integer linear program (2). Then for each $\alpha\in \mathbb{N}^n$, $\tau_{\alpha}(B)=\nu_{\alpha}(B)$ if and only if the matrix $A$ has one of the following forms: 
\begin{enumerate}
    \item $A=M_t(P_n), n\geq t$. Note that $M_t(P_n)$ is the incidence matrix corresponding to the $t$-connected ideal of the path graph, given by 
    \[
        (M_t(P_n))_{ij}=
        \begin{cases}
        1 & \text{if } j \le i \le j+t-1,\\
        0 & \text{otherwise},
        \end{cases}
    \]
    where $1\le i\le n$ and $1\le j\le n-t+1$.

    \item $A=M_t(C_n)$ with $t=n$ or $(t,n)\in \{(3,6),(3,9), (4,8)\}$. Note that $M_t(C_n)$ is the incidence matrix corresponding to the $t$-connected ideal of the cycle $C_n$, given by
    \[
    (M_t(C_n))_{ij}=
    \begin{cases}
    1 & \text{if } 0 \le (i-j \pmod n) \le t-1, \\
    0 & \text{otherwise}.
    \end{cases}
    \]
    
\end{enumerate}

\end{theorem}

\section*{Acknowledgement}
Experiments carried out using the computer algebra software Macaulay2~\cite{M2} and the package~\cite{SymbolicPowersSource} have provided numerous valuable insights. The first named author is supported by ANRF National Postdoctoral Fellowship. The authors acknowledge the support from the Infosys Foundation.

\bibliographystyle{abbrv}
\bibliography{ref} 

@incollection {Gimnez-et.al2018,
    AUTHOR = {Gimenez, Philippe and Mart\'inez-Bernal, Jos\'e{} and Simis,
              Aron and Villarreal, Rafael H. and Vivares, Carlos E.},
     TITLE = {Symbolic powers of monomial ideals and {C}ohen-{M}acaulay
              vertex-weighted digraphs},
 BOOKTITLE = {Singularities, algebraic geometry, commutative algebra, and
              related topics},
     PAGES = {491--510},
 PUBLISHER = {Springer, Cham},
      YEAR = {2018},
      ISBN = {978-3-319-96826-1; 978-3-319-96827-8},
   MRCLASS = {13F20 (05C20 05C25 13A30)},
  MRNUMBER = {3839809},
MRREVIEWER = {Elo\'isa\ Grifo},
}

@article {Grisalde.et.al2024,
    AUTHOR = {Grisalde, Gonzalo and Mart\'inez-Bernal, Jos\'e{} and
              Villarreal, Rafael H.},
     TITLE = {Normally torsion-free edge ideals of weighted oriented graphs},
   JOURNAL = {Comm. Algebra},
  FJOURNAL = {Communications in Algebra},
    VOLUME = {52},
      YEAR = {2024},
    NUMBER = {4},
     PAGES = {1672--1685},
      ISSN = {0092-7872,1532-4125},
   MRCLASS = {13C70 (05C22 05C25 05E40 13A70 13F20)},
  MRNUMBER = {4715473},
MRREVIEWER = {Thanh\ Vu},
       DOI = {10.1080/00927872.2023.2270058},
       URL = {https://doi.org/10.1080/00927872.2023.2270058},
}

@article {MandalPradhan2021,
    AUTHOR = {Mandal, Mousumi and Pradhan, Dipak Kumar},
     TITLE = {Symbolic powers in weighted oriented graphs},
   JOURNAL = {Internat. J. Algebra Comput.},
  FJOURNAL = {International Journal of Algebra and Computation},
    VOLUME = {31},
      YEAR = {2021},
    NUMBER = {3},
     PAGES = {533--549},
      ISSN = {0218-1967,1793-6500},
   MRCLASS = {05C25 (05C38 05E40 13D02 13D45)},
  MRNUMBER = {4260191},
MRREVIEWER = {Aryampilly\ V.\ Jayanthan},
       DOI = {10.1142/S0218196721500260},
       URL = {https://doi.org/10.1142/S0218196721500260},
}

@article {MandalPradhan2022,
    AUTHOR = {Mandal, Mousumi and Pradhan, Dipak Kumar},
     TITLE = {Comparing symbolic powers of edge ideals of weighted oriented
              graphs},
   JOURNAL = {J. Algebraic Combin.},
  FJOURNAL = {Journal of Algebraic Combinatorics. An International Journal},
    VOLUME = {56},
      YEAR = {2022},
    NUMBER = {2},
     PAGES = {453--474},
      ISSN = {0925-9899,1572-9192},
   MRCLASS = {05C22 (05C25 05C38 05E40)},
  MRNUMBER = {4467962},
MRREVIEWER = {Ko-Wei\ Lih},
       DOI = {10.1007/s10801-022-01118-1},
       URL = {https://doi.org/10.1007/s10801-022-01118-1},
}

@article {CooperEtAl2017,
    AUTHOR = {Cooper, Susan M. and Embree, Robert J. D. and H\`a, Huy T\`ai
              and Hoefel, Andrew H.},
     TITLE = {Symbolic powers of monomial ideals},
   JOURNAL = {Proc. Edinb. Math. Soc. (2)},
  FJOURNAL = {Proceedings of the Edinburgh Mathematical Society. Series II},
    VOLUME = {60},
      YEAR = {2017},
    NUMBER = {1},
     PAGES = {39--55},
      ISSN = {0013-0915,1464-3839},
   MRCLASS = {13F20 (13A02 14N05)},
  MRNUMBER = {3589840},
MRREVIEWER = {Mike\ Janssen},
       DOI = {10.1017/S0013091516000110},
       URL = {https://doi.org/10.1017/S0013091516000110},
}

@Misc{M2,
          author = {Grayson, Daniel R. and Stillman, Michael E.},
          title = {Macaulay2, a software system for research in algebraic geometry},
          howpublished = {Available at \url{http://www.math.uiuc.edu/Macaulay2/}}
        }

@incollection {2015Dao-et.el,
    AUTHOR = {Dao, Hailong and De Stefani, Alessandro and Grifo, Elo\'isa
              and Huneke, Craig and N\'u\~nez-Betancourt, Luis},
     TITLE = {Symbolic powers of ideals},
 BOOKTITLE = {Singularities and foliations. geometry, topology and
              applications},
    SERIES = {Springer Proc. Math. Stat.},
    VOLUME = {222},
     PAGES = {387--432},
 PUBLISHER = {Springer, Cham},
      YEAR = {2018},
      ISBN = {978-3-319-73639-6; 978-3-319-73638-9},
   MRCLASS = {13F20 (14N05)},
  MRNUMBER = {3779569},
MRREVIEWER = {Cleto\ B.\ Miranda-Neto},
       DOI = {10.1007/978-3-319-73639-6\_13},
       URL = {https://doi.org/10.1007/978-3-319-73639-6_13},
}

@article {Paolini-Salvetti,
    AUTHOR = {Paolini, Giovanni and Salvetti, Mario},
     TITLE = {Weighted sheaves and homology of {A}rtin groups},
   JOURNAL = {Algebr. Geom. Topol.},
  FJOURNAL = {Algebraic \& Geometric Topology},
    VOLUME = {18},
      YEAR = {2018},
    NUMBER = {7},
     PAGES = {3943--4000},
      ISSN = {1472-2747,1472-2739},
   MRCLASS = {05E45 (20F36 52C35)},
  MRNUMBER = {3892236},
MRREVIEWER = {Priyavrat\ Deshpande},
       DOI = {10.2140/agt.2018.18.3943},
       URL = {https://doi.org/10.2140/agt.2018.18.3943},
}

@article {Alilooee-Banerjee,
    AUTHOR = {Alilooee, Ali and Banerjee, Arindam},
     TITLE = {Packing properties of cubic square-free monomial ideals},
   JOURNAL = {J. Algebraic Combin.},
  FJOURNAL = {Journal of Algebraic Combinatorics. An International Journal},
    VOLUME = {54},
      YEAR = {2021},
    NUMBER = {3},
     PAGES = {803-813},
      ISSN = {0925-9899,1572-9192},
   MRCLASS = {13F55 (05E40)},
  MRNUMBER = {4330413},
MRREVIEWER = {Ashkan\ Nikseresht},
       DOI = {10.1007/s10801-021-01020-2},
       URL = {https://doi.org/10.1007/s10801-021-01020-2},
}

@article {Bodas-et.al,
    AUTHOR = {Bodas, Hrishikesh and Drabkin, Benjamin and Fong, Caleb and
              Jin, Su and Kim, Justin and Li, Wenxuan and Seceleanu,
              Alexandra and Tang, Tingting and Williams, Brendan},
     TITLE = {Consequences of the packing problem},
   JOURNAL = {J. Algebraic Combin.},
  FJOURNAL = {Journal of Algebraic Combinatorics. An International Journal},
    VOLUME = {54},
      YEAR = {2021},
    NUMBER = {4},
     PAGES = {1095-1117},
      ISSN = {0925-9899,1572-9192},
   MRCLASS = {13F55 (05C15 05C65 05E40)},
  MRNUMBER = {4348918},
MRREVIEWER = {Adam\ L.\ Van Tuyl},
       DOI = {10.1007/s10801-021-01039-5},
       URL = {https://doi.org/10.1007/s10801-021-01039-5},
}

@article{2025Ananthnarayan-et.al,
  TITLE={Linear quotients of connected ideals of graphs},
  AUTHOR={Ananthnarayan, H and Javadekar, Omkar and Maithani, Aryaman},
  JOURNAL={Journal of Algebraic Combinatorics},
  VOLUME={61},
  NUMBER={3},
  PAGES={34},
  YEAR={2025},
  PUBLISHER={Springer}
}

@article{2025Das-et.al,
      TITLE={Stanley-{R}eisner ideals of higher independence complexes of chordal graphs.}, 
      AUTHOR={Kanoy Kumar Das and Amit Roy and Kamalesh Saha},
      YEAR={2026},
      JOURNAL={International Journal of Algebra and Computation, online ready},
      URL={https://doi.org/10.1142/S021819672650013X}, 
}

@article {2022Deshpande-et.al,
    AUTHOR = {Deshpande, Priyavrat and Shukla, Samir and Singh, Anurag},
     TITLE = {Distance {$r$}-domination number and {$r$}-independence
              complexes of graphs},
   JOURNAL = {European J. Combin.},
  FJOURNAL = {European Journal of Combinatorics},
    VOLUME = {102},
      YEAR = {2022},
     PAGES = {Paper No. 103508, 14},
      ISSN = {0195-6698,1095-9971},
   MRCLASS = {05C69 (05C12)},
  MRNUMBER = {4374797},
MRREVIEWER = {Dorota\ Kuziak},
       DOI = {10.1016/j.ejc.2022.103508},
       URL = {https://doi.org/10.1016/j.ejc.2022.103508},
}

@article{2023Abdelmalek-et.al,
  TITLE={Chordal graphs, higher independence and vertex decomposable complexes},
  AUTHOR={Abdelmalek, Fred M and Deshpande, Priyavrat and Goyal, Shuchita and Roy, Amit and Singh, Anurag},
  JOURNAL={International Journal of Algebra and Computation},
  VOLUME={33},
  NUMBER={03},
  PAGES={481--498},
  YEAR={2023},
  PUBLISHER={World Scientific}
}

@article{2024Deshpande-et.al,
  TITLE={Fr{\"o}berg’s theorem, vertex splittability and higher independence complexes},
  AUTHOR={Deshpande, Priyavrat and Roy, Amit and Singh, Anurag and Van Tuyl, Adam},
  JOURNAL={Journal of Commutative Algebra},
  VOLUME={16},
  NUMBER={4},
  PAGES={391--410},
  YEAR={2024},
  PUBLISHER={Rocky Mountain Mathematics Consortium Tempe, AZ, USA}
}

@article{2021Deshpande-et.al,
AUTHOR = {Deshpande, Priyavrat and Singh, Anurag},
YEAR = {2021},
MONTH = {03},
PAGES = {pp. 53–71},
TITLE = {Higher independence complexes of graphs and their homotopy types},
VOLUMME = {36},
JOURNAL = {Journal of the Ramanujan Mathematical Society}
}

@article{2025Ghodh-et.al,
      TITLE={Linear resolution of connected graph ideals and their powers}, 
      AUTHOR={Arka Ghosh and S Selvaraja},
      YEAR={arXiv:2512.06346},
      EPRINT={2512.06346},
      ARCHIVEPREFIX={arXiv},
      URL={https://arxiv.org/abs/2512.06346}, 
}

@article{2025Deshpande-et.al,
      TITLE={The complex of $r$-co-connected subgraphs, chordality and {Fr{\"o}berg}'s theorem}, 
      AUTHOR={Priyavrat Deshpande and Amit Roy and Rutuja Sawant},
      YEAR={arXiv:2510.25710},
      EPRINT={2510.25710},
      ARCHIVEPREFIX={arXiv},
      PRIMARYCLASS={math.CO},
      URL={https://arxiv.org/abs/2510.25710}, 
}

@article{2025Roy-et.al,
  AUTHOR    = {Amit Roy and Sourav Kanti Patra},
  TITLE     = {Graded Betti Numbers of a Hyperedge Ideal Associated to Join of Graphs},
  JOURNAL   = {Bulletin of the Malaysian Mathematical Sciences Society},
  VOLUME    = {48},
  YEAR      = {2025},
  DOI       = {10.1007/s40840-025-01897-3},
}

@article {2005Gitler-et.al,
    AUTHOR = {Gitler, Isidoro and Valencia, Carlos and Villarreal, Rafael
              H.},
     TITLE = {A note on the {R}ees algebra of a bipartite graph},
   JOURNAL = {J. Pure Appl. Algebra},
  FJOURNAL = {Journal of Pure and Applied Algebra},
    VOLUME = {201},
      YEAR = {2005},
    NUMBER = {1-3},
     PAGES = {17--24},
      ISSN = {0022-4049,1873-1376},
   MRCLASS = {13A30 (13F20)},
  MRNUMBER = {2158744},
MRREVIEWER = {Viviana\ Ene},
       DOI = {10.1016/j.jpaa.2004.12.013},
       URL = {https://doi.org/10.1016/j.jpaa.2004.12.013},
}

@article{1990Cornejols-et.al,
AUTHOR = {G{\'e}rard Cornu{\'e}jols
 and Michele Conforti},
TITLE = {A decomposition theorem for balanced matrices, Integer Programming and Combinatorial Optimization},
YEAR = {1990},
VOLUME = {74},
JOURNAL = {},
PAGES = {147-169},
}

@article {2025Banerjee-et.al,
    AUTHOR = {Banerjee, Arindam and Das, Kanoy Kumar and Selvaraja, S.},
     TITLE = {Ordinary and symbolic powers of edge ideals of weighted
              oriented graphs},
   JOURNAL = {Bull. Malays. Math. Sci. Soc.},
  FJOURNAL = {Bulletin of the Malaysian Mathematical Sciences Society},
    VOLUME = {48},
      YEAR = {2025},
    NUMBER = {3},
     PAGES = {Paper No. 85, 22},
      ISSN = {0126-6705,2180-4206},
   MRCLASS = {13F20 (05C22 05E40 13A70)},
  MRNUMBER = {4889334},
       DOI = {10.1007/s40840-025-01842-4},
       URL = {https://doi.org/10.1007/s40840-025-01842-4},
}

@article{2024Kannoy,
AUTHOR = {Kanoy Kumar Das},
TITLE = {Equality of ordinary and symbolic powers of some classes of monomial ideals}, 
JOURNAL = {Graphs and Combinatorics},
YEAR = {2024},
VOLUME = {40(12)},
}

@article {2023Banerjee-et.al,
    AUTHOR = {Banerjee, Arindam and Chakraborty, Bidwan and Das, Kanoy Kumar
              and Mandal, Mousumi and Selvaraja, S.},
     TITLE = {Equality of ordinary and symbolic powers of edge ideals of
              weighted oriented graphs},
   JOURNAL = {Comm. Algebra},
  FJOURNAL = {Communications in Algebra},
    VOLUME = {51},
      YEAR = {2023},
    NUMBER = {4},
     PAGES = {1575--1580},
      ISSN = {0092-7872,1532-4125},
   MRCLASS = {13F20 (05C22 05E40)},
  MRNUMBER = {4552913},
MRREVIEWER = {Monica\ La Barbiera},
       DOI = {10.1080/00927872.2022.2139834},
       URL = {https://doi.org/10.1080/00927872.2022.2139834},
}

@article{2025Roy-Saha,
AUTHOR = {Amit Roy and Kamalesh Saha},
TITLE = {Equality of ordinary and symbolic powers and the Conforti-Cornu\'ejols conjecture for $(n-2)$-uniform clutters}, 
YEAR = {2025},
EPRINT = {2510.15864},
ARCHIVEPREFIX = {arXiv},
PRIMARYCLASS={math.AC},
URL = {https://arxiv.org/abs/2510.15864}, 
}

@article {2025Pinto-et.al,
    AUTHOR = {M\'endez, Fernando O. and Vaz Pinto, Maria and Villarreal,
              Rafael H.},
     TITLE = {Symbolic powers: {S}imis and weighted monomial ideals},
   JOURNAL = {J. Algebra Appl.},
  FJOURNAL = {Journal of Algebra and its Applications},
    VOLUME = {24},
      YEAR = {2025},
    NUMBER = {13-14},
     PAGES = {Paper No. 2541001, 28},
      ISSN = {0219-4988,1793-6829},
   MRCLASS = {13C70 (05C22 05E40 13A70 13F20)},
  MRNUMBER = {4994910},
       DOI = {10.1142/S0219498825410014},
       URL = {https://doi.org/10.1142/S0219498825410014},
}

@article {2009Gitler-et.al,
    AUTHOR = {Gitler, Isidoro and Reyes, Enrique and Villarreal, Rafael H.},
     TITLE = {Blowup algebras of square-free monomial ideals and some links
              to combinatorial optimization problems},
   JOURNAL = {Rocky Mountain J. Math.},
  FJOURNAL = {The Rocky Mountain Journal of Mathematics},
    VOLUME = {39},
      YEAR = {2009},
    NUMBER = {1},
     PAGES = {71--102},
      ISSN = {0035-7596,1945-3795},
   MRCLASS = {13A30 (13B22 13F20 13H10 52B20)},
  MRNUMBER = {2476802},
MRREVIEWER = {John\ Clark},
       DOI = {10.1216/RMJ-2009-39-1-71},
       URL = {https://doi.org/10.1216/RMJ-2009-39-1-71},
}

@article{2024Nasernejad,
AUTHOR = {W. Hochst{\"a}ttler and M. Nasernejad},
TITLE = {A classification of mengerian 4-uniform hypergraphs derived from
graphs},
JOURNAL = {Ars Combinatoria},
YEAR = {2024},
VOLUME = {161},
PAGES = {49-59},
}

@article {1994Simis-et.al,
    AUTHOR = {Simis, Aron and Vasconcelos, Wolmer V. and Villarreal, Rafael
              H.},
     TITLE = {On the ideal theory of graphs},
   JOURNAL = {J. Algebra},
  FJOURNAL = {Journal of Algebra},
    VOLUME = {167},
      YEAR = {1994},
    NUMBER = {2},
     PAGES = {389--416},
      ISSN = {0021-8693,1090-266X},
   MRCLASS = {13A30 (05C90 13H10)},
  MRNUMBER = {1283294},
MRREVIEWER = {P.\ Schenzel},
       DOI = {10.1006/jabr.1994.1192},
       URL = {https://doi.org/10.1006/jabr.1994.1192},
}

@article{2025Ficarra-Moradi,
AUTHOR = {A. Ficarra and S. Moradi},
TITLE = {Symbolic powers of polymatroidal ideals},
JOURNAL = {Journal of Pure and Applied Algebra},
YEAR = {2025},
VOLUME = {229(10)},
PAGES = {108082},
}

@misc{SymbolicPowersSource,
  title = {{SymbolicPowers: A \emph{Macaulay2} package. Version~2.0}},
  author = {Eloisa Grifo},
  howpublished = {A \emph{Macaulay2} package available at
    \url{https://github.com/Macaulay2/M2/tree/stable/M2/Macaulay2/packages}}
}


\end{document}